\documentclass[12pt]{amsart}

\usepackage{amsfonts, amssymb, amscd}

\def\dual                 {{\vee}}

\def\ee                 {{\rm e}}
\def\MSV		{{\rm MSV}}

\def\ZZ                 {{\mathbb Z}}
\def\PP                {{\mathbb P}}
\def\RR                 {{\mathbb R}}
\def\CC                 {{\mathbb C}}
\def\QQ                 {{\mathbb Q}}

\newtheorem{lemma}{Lemma}[section]
\newtheorem{theorem}[lemma]{Theorem}

\newtheorem{proposition}[lemma]{Proposition}
\theoremstyle{definition}
\newtheorem{definition}[lemma]{Definition}

\newtheorem{remark}[lemma]{Remark}
\theoremstyle{remark}
\newtheorem*{proof*}{Proof}
\numberwithin{equation}{section}
\begin{document}
\title{On CY-LG correspondence for (0,2) toric models}

\author{Lev A.  Borisov}
\address{Rutgers University, Department of Mathematics, 110 Frelinghuysen Rd.,
Piscataway \\ NJ \\ 08854 \\ USA}
\email{borisov@math.rutgers.edu}
\author{Ralph M. Kaufmann}
\address{Purdue University, Department of Mathematics, 150 N. University St.,
West Lafayette \\ IN \\47907\\USA}
\email{rkaufman@math.purdue.edu}

\begin{abstract}
We conjecture a description of the vertex (chiral) algebras of the (0,2) nonlinear sigma
models on smooth quintic threefolds. We provide evidence in favor of the conjecture
by connecting our algebras to the cohomology of a twisted chiral de Rham sheaf.
We discuss CY/LG correspondence in this setting.
\end{abstract}

\maketitle

\section{Introduction}
The goal of this paper is to show that the vertex algebra approach to 
toric mirror symmetry is suitable for working with the (0,2) theories.
Compared to their (2,2) cousins,
(0,2) nonlinear sigma models are poorly understood. There 
has been a renewed recent interest in them, see for example \cite{Guffin}.
This paper aims to provide a concrete tool for various calculations in
the theories. We focus our attention on the quintic case, but most of our techniques 
are applicable in a much wider context. 

\smallskip
Let us review the basics of the vertex algebra approach to mirror symmetry.
In the very important paper \cite{MSV} Malikov, Schechtman and Vaintrob
have constructed the so called chiral de Rham complex, which is a sheaf
of vertex (in physics literature \emph{chiral}) algebras over a given smooth
manifold $X$. Its cohomology should be viewed as the large K\"ahler limit
of  the space of states 
of the half-twisted theory for the type II string models with target $X$,
see \cite{WK}\footnote{There is an alternative interpretation of chiral de Rham 
complex in the works of Heluani and coathors, see for example \cite{Heluani}. We
thank the referee for pointing this out to us.}.

\smallskip
The chiral de Rham complex  $\MSV(X)$ is defined locally. Thus, it does not 
carry the information about instanton corrections. It is expected that
one should be able (in the simply connected case) to construct a deformation
of its cohomology that would incorporate these corrections, along the lines
of the construction of quantum cohomology. However, this construction
is not presently known. 

\smallskip
In the case when $X$ is a hypersurface in a Fano toric variety, an ad hoc
deformation has been defined in \cite{Borvert}, motivated by Batyrev's
mirror symmetry. Specifically, let $M_1$ and $N_1$ be dual lattices (in
this paper this simply means free abelian groups), and let
$\Delta$ and $\Delta^\dual$ be dual reflexive polytopes in them. Consider
extended dual lattices $M=M_1\oplus \ZZ$ and $N=N_1\oplus \ZZ$
and cones $K=\RR_{\geq 0}(\Delta,1) \cap M$ and 
$K^\dual=\RR_{\geq 0}(\Delta^\dual,1) \cap N$ in them.
Then the vertex algebras of mirror symmetry are defined  in  \cite{Borvert}
as the cohomology of the lattice vertex algebra ${\rm Fock}_{M\oplus N}$ 
by the differential 
$$
D_{f,g}={\rm Res}_{z=0}\Big(\sum_{m\in\Delta} f_m m^{ferm}(z)\ee^{\int m^{bos}(z)}
+\sum_{n\in\Delta^\dual} g_n n^{ferm}(z)\ee^{\int n^{bos}(z)}\Big)
$$
where $f_m$ and $g_n$ are complex parameters.
This construction may be extended to a more general setting of Gorenstein
dual cones. The resulting algebras have numerous nice properties, studied in
\cite{chiralrings}. In particular, they admit $N=2$ structures and their 
chiral rings can be calculated. This approach is somewhat different from the gauged linear sigma model approach of \cite{Witten} since it is based on the classical description of toric varieties in terms of their fans, as opposed to the homogeneous coordinate ring construction
of Cox.
 
 \smallskip
This paper is dealing with  a certain generalization the theory known 
as (0,2) nonlinear sigma model. One major difference is that 
the tangent bundle $TX$ is replaced 
by another vector bundle $E$ with the same first and second
Chern classes. The influential paper of Witten
\cite{Witten} describes such theories for the case of the hypersurfaces in
the projective space. In this paper we will specifically focus on the quintic threefolds
in $\PP^4$,
although our techniques are valid in any dimension.

\smallskip
As in \cite[(6.39-40)]{Witten}, we consider a homogeneous polynomial $G$ of 
degree $5$ in the homogeneous coordinates $x_i$ on $\PP^4$ 
and five polynomials $G^i$ of degree four in these coordinates with the property
$\sum_i x_iG^i=0$. Equivalently, we consider 
five polynomials of degree four $R^i=\partial_iG+G^i$.
Witten has constructed  (physically)
a one-dimensional family of (0,2) theories that interpolates
between the Calabi-Yau and the Landau-Ginzburg phases.
The Calabi-Yau theory in question is defined
by the quintic $G=0$, but with a vector bundle that is a deformation of the tangent 
bundle, given by $G^i$. 
We argue that the half-twisted theories for these data are given
by the cohomology 
of the lattice vertex algebra ${\rm Fock}_{M\oplus N}$
by the differential
$$
D_{(F^\cdot),g}={\rm Res}_{z=0}\Big(\sum_{\stackrel{m\in \Delta}{0\leq i\leq 4}}F^i_m m_i^{ferm}(z)\ee^{\int m^{bos}(z)}
+\sum_{n\in\Delta^\dual} g_n n^{ferm}(z)\ee^{\int  n^{bos}(z)}\Big)
$$
where $F^i=x_iR^i$ are degree $5$ polynomials that generalize the logarithmic
derivatives of the equation of the quintic (see Section \ref{sechet} for details).
Equivalently, one can take the cohomology 
of  ${\rm Fock}_{M\oplus K^\dual}$ by the above differential $D_{(F^\cdot),g}$.
We denote these vertex algebras by $V_{(F^\cdot),g}$. In the case when $F^i=x_i\partial_if$
are logarithmic derivatives of some degree five polynomial $f$, we have 
$V_{(F^\cdot),g}=V_{f,g}$, i.e. these algebras generalize the usual vertex algebras of mirror 
symmetry.

\smallskip
We consider a natural "limit" of the algebras $V_{(F^\cdot),g}$ for fixed $F^i$, given by 
the cohomology of the so-called partial (\emph{deformed} in \cite{Borvert})
lattice vertex algebra ${\rm Fock}_{M\oplus N}^\Sigma$ by the above 
differential $D_{(F^\cdot),g}$. Our main result is Theorem \ref{5.1}.

\medskip\noindent
{\bf Theorem \ref{5.1}.} 
The cohomology of  ${\rm Fock}_{M\oplus K^\dual}^\Sigma$ with respect to 
$D_{(F^\cdot),g}$ is isomorphic to the cohomology of a twisted chiral de Rham 
sheaf on the quintic $\sum_{i=0}^4F^i=0$ given by $R^i$.
\medskip

The twisted chiral de Rham sheaf in question is the one studied in 
\cite{gerbes1,gerbes2,gerbes3}. It appears that our construction provides,
rather unexpectedly, a specific choice among such sheaves, which was pointed to us 
by Malikov. In another limit we expect to see the Landau-Ginzburg
phase of the theory. Thus, the CY/LG correspondence considered in \cite{Witten}
is manifest in our construction.

\smallskip
The paper is organized as follows. In Section \ref{sec1}, we 
recall the construction of \cite{Borvert} as it applies to the case of quintics in
$\PP^4$. We recall the Calabi-Yau -- Landau-Ginzburg correspondence in this
setting. In Section \ref{sechet}, we define the vertex algebras for the (0,2) sigma 
model of the quintic, see Definition \ref{hetd}. Section \ref{seccoh}
is devoted to the proof of the technical result Theorem \ref{tricky} which 
is necessary to apply the method of \cite{Borvert} to this setting. Theorem \ref{tricky}
may be 
of independent interest, as it gives a novel way of constructing a twisted
chiral de Rham sheaf in some cases.
In Section \ref{CYLG}, we prove the main Theorem \ref{5.1}.
In Section \ref{chiral},
we discuss further properties of the vertex algebras for (0,2) models
on the quintic that follow from the techniques of \cite{Borvert} and \cite{chiralrings}.
Specifically, we focus on the description of their chiral rings.
Finally, in Section \ref{last} we sketch some future directions of research.

\smallskip
{\bf Acknowledgements}. We thank Fyodor Malikov for insightful comments
on the preliminary version of the paper. LB thanks Ron Donagi for directing his attention to the 
topic.  LB's work was  supported  by NSF DMS-1003445.
RK thankfully acknowledges
support from NSF DMS--0805881. He also would like to thank the
Institute for Advanced Study for its support during the project.
While at the IAS, RK's work was supported by the NSF under agreement
DMS--0635607.
Any opinions, findings and conclusions or
recommendations expressed in this
material are those of the authors and do not necessarily
reflect the views of the National Science Foundation.

\section{Overview of vertex operator algebras of mirror symmetry for  the quintic}
\label{sec1}

For a smooth manifold $X$, the chiral de Rham complex $\MSV(X)$ is 
a sheaf of vertex algebras on $X$ constructed in \cite{MSV}. In a given 
coordinate system near a point on $X$ this sheaf is 
generated by $4\dim X$ free fields $b^i$, $\phi^i$, $\psi_i$, $a_i$ with 
the operator product expansions (OPEs) 
$$
a_i(z)b^j(w)\sim \delta_i^j(z-w)^{-1},~~\phi^i(z)\psi_j(w)\sim\delta^i_j(z-w)^{-1}
$$
and all the others nonsingular. Here the fields $a$ and $b$ are bosonic and fields
$\phi$ and $\psi$ are fermionic. The $b$ fields transform like coordinates
on $X$. Products of $b$ and $\phi$ transform under the coordinate changes 
as differential $k$-forms (where $k$ 
is the number of $\phi$ factors). Products of $b$ and $\psi$ transform as 
polyvector fields.

\smallskip
The sheaf ${\rm MSV}(X)$ carries a natural conformal structure, in fact
it contains a natural $N=1$ algebra in it. If, in addition,
$X$ is a Calabi-Yau manifold, then depending on a choice of nowhere vanishing
holomorphic volume form (up to constant), the $N=1$ structure can be extended to
$N=2$ structure, see \cite{MSV}.

\smallskip
For a manifold $X$, the cohomology $H^*(\MSV(X))$ of the chiral de Rham complex
on it provides a fascinating invariant. It inherits the vertex algebra structure from the chiral de Rham complex. Its natural $N=1$ structure is extended to
a natural $N=2$ structure when $X$ is a Calabi-Yau (in fact if $X$ is in addition compact,
then the choice of the volume form is unique up to scaling, so the $N=2$ structure
is canonically defined). From the string theory point of view
$H^*(\MSV(X))$
can be thought of as a large K\"ahler limit of the space of half-twisted type
II string theory with target $X$, see \cite{WK}.

\smallskip
We will now review the (fairly) explicit description 
of the cohomology of the chiral de Rham complex
for a smooth quintic in $\PP^4$, which was obtained in \cite{Borvert}. We 
will also describe the cohomology of the 
chiral de Rham complex  for the canonical bundle $W$ over 
$\PP^4$. 

\smallskip
Consider the dual lattices $M$ and $N$ defined as 
$$
M:=\{(a_0,\ldots,a_4)\in \ZZ^5,\sum a_i =0 \hskip -10pt\mod 5\};
~~~
N:=\ZZ^5+\ZZ(\frac 15,\ldots, \frac 15)
$$
with the usual dot product pairing.
We introduce elements $\deg=(1,\ldots, 1)\in M$ and $\deg^\dual =(\frac 15,\ldots, \frac 15)$ 
in $N$.

\smallskip
The cone $K$ in $M$ is defined by the inequalities $a_i\geq 0$. The intersection of $K$ with the hyperplane 
$\bullet\cdot \deg^\dual=1$ is the polytope $\Delta\in M$. This is a four-dimensional simplex which is the convex 
hull of $(5,0,0,0,0),\ldots,(0,0,0,0,5)$. The
 dual cone $K^\dual$ in $N$ is also defined by nonnegativity of the coordinates.
The polytope $\Delta^\dual = K^\dual\cap \{\deg\cdot\bullet = 1\}$ is the simplex with vertices $(1,0,0,0,0),\ldots,(0,0,0,0,1)$. The only other lattice point of $\Delta^\dual$ is $\deg^\dual$. 

\begin{remark}
The lattice points in $\Delta$ correspond to monomials of degree $5$ in homogeneous coordinates on  $\PP^4$
while the lattice points in $\Delta^\dual$ correspond to codimension one
torus strata on the canonical bundle $W$ over $\PP^4$.
\end{remark}

We will now describe briefly the construction of the vertex algebras ${\rm Fock}_{M\oplus N}$
and ${\rm Fock}_{M\oplus N}^\Sigma$, following \cite{Borvert}.
We start with the vertex algebra ${\rm Fock}_{0\oplus 0}$ generated by $10$ free bosonic and $10$ free fermionic
fields based on the lattice $M\oplus N$ with operator product expansions
$$
m^{bos}(z)n^{bos}(w) \sim \frac {m\cdot n }{(z-w)^2},~~
m^{ferm}(z)n^{ferm}(w) \sim \frac {m\cdot n }{(z-w)}
$$
and all other OPEs nonsingular. 
We then consider the lattice vertex algebra ${\rm Fock}_{M\oplus N}$ with
additional vertex operators 
$\ee^{\int m^{bos}(z)+n^{bos}(z)}$ (with the appropriate cocycle, see \cite{Borvert}).
They satisfy
\begin{equation}\label{vertop}
\begin{array}{c}
\ee^{\int m_1^{bos}(z)+n_1^{bos}(z)}
\ee^{\int m_2^{bos}(w)+n_2^{bos}(w)}\\
=(z-w)^{m_1\cdot n_2+m_2\cdot n_1 }
\ee^{\int m_1^{bos}(z)+n_1^{bos}(z)+m_2^{bos}(w)
+n_2^{bos}(w)}
\end{array}
\end{equation}
with the normal ordering implicitly applied. Here the right hand side needs to
be expanded at $z=w$.

\smallskip
Consider the (generalized) fan $\Sigma$ in $N$ given as follows. Its maximum-dimensional cones are 
generated by $\deg^\dual, -\deg^\dual$ and four out of the five vertices of $\Delta^\dual$. It is the preimage
in $N$ of the fan of $\PP^4$ given by the images of the generators of $\Delta^\dual$ in $N/\ZZ\deg^\dual$.
Then define the partial lattice vertex algebra ${\rm Fock}_{M\oplus N}^\Sigma$ by setting the 
product in \eqref{vertop} to zero
if $n_1$ and $n_2$ do not lie in the same cone of $\Sigma$.
We similarly define the vertex algebras ${\rm Fock}_{M\oplus K^\dual}$ and ${\rm Fock}_{M\oplus K^\dual}^\Sigma$.

\smallskip
The following results have been proved in \cite{Borvert}.
\begin{proposition}
Let $W\to \PP^4$ be the canonical bundle.
Then the cohomology of the chiral de Rham complex $\MSV(W)$
is isomorphic to
the cohomology of ${\rm Fock}_{M\oplus K^\dual}^\Sigma$ with respect to the differential
$$D_g={\rm Res}_{z=0}\sum_{n\in\Delta^\dual} g_n n^{ferm}(z)\ee^{\int n^{bos}(z)}$$
for any collection of nonzero numbers $g_n, n\in \Delta^\dual$.
\end{proposition}

\begin{proposition}
The cohomology of the chiral de Rham complex of a smooth quintic $F(x_0,\ldots, x_4)=0$ which is 
transversal to the torus strata is given 
by the cohomology of ${\rm Fock}_{M\oplus K}^\Sigma$ by the differential 
$$D_{f,g}={\rm Res}_{z=0}\Big(\sum_{m\in\Delta} f_m m^{ferm}(z)\ee^{\int m^{bos}(z)}
+\sum_{n\in\Delta^\dual} g_n n^{ferm}(z)\ee^{\int n^{bos}(z)}\Big)$$
where $g_n$ are arbitrary nonzero numbers and $f_m$ is the coefficient of $F$ by the corresponding
monomial. 
\end{proposition}

The cohomology of the chiral de Rham complex should be viewed as just an approximation to
the true physical vertex algebra of  the half-twisted theory. It has been conjectured
in \cite{Borvert} that the effect of adding instanton corrections to this algebra must 
correspond to the removal of the superscript
$^\Sigma$ in the calculation of the cohomology. Crucially,
while the cohomology of ${\rm Fock}_{M\oplus K^\dual}^\Sigma$ with respect to  $D_{f,g}$ is independent from $g$ (as long as all $g_n$ are nonzero), the cohomology of ${\rm Fock}_{M\oplus 
K^\dual}$ with respect to $D_{f,g}$ depends on it. 

\begin{definition}\label{vfg} 
Fix $F$ and the corresponding $f_m$. As the $g_n$ vary, consider the family of vertex 
algebras $V_{f,g}$ which are the cohomology of ${\rm Fock}_{M\oplus K^\dual}$ with respect to
 the differential
$$D_{f,g}={\rm Res}_{z=0}\Big(\sum_{m\in\Delta} f_m m^{ferm}(z)\ee^{\int m^{bos}(z)}
+\sum_{n\in\Delta^\dual} g_n n^{ferm}(z)\ee^{\int n^{bos}(z)}\Big).$$
We call this a family of vertex algebras of mirror symmetry associated to the quintic $F=0$.
\end{definition}

The vertex algebras of mirror symmetry provide a useful way of thinking about the so-called 
Calabi-Yau -- Landau-Ginzburg (CY-LG)
correspondence for the $N=2$ theories related to the quintic, which we describe below.

\smallskip
There are a priori six parameters in the Definition \ref{vfg},  that correspond to the values of 
$g_n$ for $n=\deg^\dual$ or the vertices $v_i$ of the simplex $\Delta^\dual$. However, up
to torus symmetry, the algebra depends only on $(\prod_{i} g_{v_i})/g_{\deg^\dual}^5$,
where $v_i$ are the vertices of $\Delta^\dual$.
Indeed, for any linear function $r:N\to \CC$ one can rescale 
$\ee^{\int n^{bos}(z)}$  to $\ee^{r(n)}\ee^{\int n^{bos}(z)}$. This will not change the OPEs
of any fields in question. This shows that the collection $g_n$ can be replaced
by $g_n \ee^{r(n)}$ for any $r$.

\smallskip
Let us now pick a piecewise-linear real-valued 
function $\rho$ which is strongly convex on $\Sigma$.
If we rescale  $\ee^{\int n^{bos}(z)}$  to  $\ee^{\lambda \rho(n)}\ee^{\int n^{bos}(z)}$
for $\lambda\to\infty$, we see that the OPEs of the new vertex operators start to approach
those for ${\rm Fock}_{M\oplus K^\dual}^\Sigma$. This implies that as the 
ratio $(\prod_{i} g_{v_i})/g_{\deg^\dual}^5$. approaches $0$, the vertex algebras of 
mirror symmetry approach (in some rather weak sense) the cohomology of the chiral
de Rham complex on the quintic. 
Specifically, while it is not known if the family of algebras stays flat after taking 
the quotient by $D_{f,g}$,
it is still reasonable to think of the cohomology of chiral de Rham complex of $F=0$ as a 
limit of $V_{f,g}$. Similarly, as this ratio approaches to $0$ one gets to the so-called 
orbifold point on the K\"ahler moduli  space of the theory, which is in the Landau-Ginzburg 
region of the moduli space. While the $D_{f,g}$ cohomology in fact jumps at the orbifold
point (see \cite{Malikov-Gorbounov}), we still want to think of the family 
$V_{f,g}$ 
as interpolating between the Calabi-Yau and the Landau-Ginzburg phases of the theory. 

\section{Vertex algebras of (0,2) nonlinear sigma models for the quintic}\label{sechet}
In the influential paper \cite{Witten} Witten has, in particular, considered a CY-LG correspondence for some (0,2) models. 
The key observation of our paper is that we can very naturally 
modify the vertex algebras of mirror symmetry for the quintic to accommodate this 
larger class of theories. The goal of this section is to give a definition of the vertex
algebras of the (0,2) sigma models for the quintic, analogous to Definition \ref{vfg}.

\smallskip
Specifically, in \cite[(6.39-40)]{Witten} Witten considered a homogeneous polynomial $G$ of 
degree $5$ in the variables $x_i$ and five polynomials $G^i$ in variables $x_i$ with
$\sum_i x_iG^i=0$ and has constructed  (physically)
a one-dimensional family of theories that interpolates
from the Calabi-Yau to the Landau-Ginzburg phases. The Calabi-Yau theory in question is defined
by the quintic $G=0$, but with the vector bundle that is a deformation of the tangent bundle, given
by $G^i$. 

\smallskip
Clearly, the above data are equivalent to a collection of five polynomials of degree four in $x_i$
which are given by $R^i=\partial_i G + G^i$. Indeed, $G$ can then be uniquely recovered as 
$\frac 15\sum_i x_iR^i$. Equivalently, we may consider five polynomials $F^i=x_iR^i$ which 
are of degree $5$ with the property that $F^i\vert_{x_i=0}=0$. In this language, the quintic
is simply $\sum_i F^i=0$.

\begin{definition}\label{hetd}
As in Section \ref{sec1} consider the vertex algebra ${\rm Fock}_{M\oplus K^\dual}$.
Define by $m_i$ the basis of $M_\QQ$ which is dual to the basis of $N_\QQ$ given by the 
vertices of $\Delta^\dual$. Consider the differential
$$D_{(F^\cdot),g}={\rm Res}_{z=0}\Big(\sum_{\stackrel{m\in \Delta}{0\leq i\leq 4}}F^i_m m_i^{ferm}(z)\ee^{\int m^{bos}(z)}
+\sum_{n\in\Delta^\dual} g_n n^{ferm}(z)\ee^{\int  n^{bos}(z)}\Big)
$$
where $g_n$ are six generic complex numbers and $F^i_m$ is the coefficient of the monomial
of degree $5$ of $F^i$ that corresponds to $m$. We call the corresponding cohomology spaces  $V_{(F^\cdot),g}$
the vertex algebras of the (0,2) sigma model on $\sum_iF^i=0$.
\end{definition}

The above definition implicitly assumes that $D_{(F^\cdot),g}$ is a differential, but this requires a verification. 
\begin{proposition}\label{diff}
The above-defined $D_{(F^\cdot),g}$ is a differential and the cohomology inherits the structure of a vertex algebra.
\end{proposition}

\begin{proof}
We need to show that all modes of the corresponding field of the algebra anti-commute with each other.
This means verifying that the OPEs of $F^i_m m_i^{ferm}(z)\ee^{\int m^{bos}(z)}$ and
$g_n n^{ferm}(z)\ee^{\int n^{bos}(z)}$ with each other and themselves are nonsingular. The only 
interesting cases are the OPEs between the above two operators. There are three possibilities:  
 $n=\deg^\dual$, $n$ is a vertex of $\Delta^\dual$ that corresponds to $i$ and $n$ is
some other vertex.

\smallskip
Case 1: $n=\deg^\dual$. Because $m\cdot \deg^\dual=1$,
the OPE of the bosonic terms $\ee^{\int m^{bos}(z)}$ and $\ee^{\int n^{bos}(z)}$
will start with $(z-w)^1$, which counteracts the $(z-w)^{-1}$ from the fermionic terms.

\smallskip
Case 2: $n$ is a vertex of $\Delta^\dual$ equal to $i$. Because $F^i\vert_{x_i=0}=0$, we may
assume that $m$ corresponds to a monomial that is divisible by $x_i$. Thus, $m\cdot n\geq 1$
and we proceed as in the previous case.

\smallskip
Case 3: $n$ is some other vertex of $\Delta^\dual$.  Then $m_i\cdot n=0$ and the fermionic
OPE has no pole at $z=w$. The bosonic OPE has no pole either, because $m\cdot n\geq 0$.
Thus the OPE is nonsingular.
\end{proof}

\begin{remark}
If one uses the same $N$-part of the differential $D_{f,g}$ but attempts to
consider various elements of $M^{ferm}(z) \ee^{\int \Delta^{bos}(z)}$ for the $M$-part, the 
condition of being a differential is equivalent to it being given by Definition \ref{hetd} for
some $F^i$ with $F^i\vert_{x_i=0}=0$.
\end{remark}
 
\begin{remark}
In the original setting of the vertex algebras of mirror symmetry, the cohomology 
with respect to $D_{f,g}$ 
inherited an $N=2$ structure from ${\rm Fock}_{M\oplus K^\dual}$ which was generated
by the fields $M^{ferm}\cdot N^{bos} - \partial_z \deg^{ferm}$ and $M^{bos}\cdot N^{ferm}-\partial_z
(\deg^\dual)^{ferm}$. Typically, this structure does not super-commute with the differential
$D_{(F^\cdot),g}$ and thus does not descend to the cohomology $V_{(F^\cdot),g}$. However, 
part of the structure still descends, as is shown below.
\end{remark}

\begin{proposition}\label{3.5}
Consider the Virasoro algebra and affine $U(1)$ algebras on ${\rm Fock}_{M\oplus K^\dual}$
which are given by 
$$
L(z):=\sum_i m_i^{bos}n_i^{bos}+\sum_i (\partial_zm_i^{ferm})n_i^{ferm}-\partial_z (\deg^\dual)^{bos}
$$
$$
J(z):=\sum_i m_i^{ferm}n_i^{ferm}+\deg^{bos}-(\deg^\dual)^{bos}.
$$
Here $m_i$ and $n_i$ are elements of a dual basis.
These fields commute with $D_{(F^\cdot),g}$ and thus descend to $V_{(F^\cdot),g}$.
\end{proposition}

\begin{proof}
The parts of the differential that correspond to  $n\in\Delta^\dual$ have already been
considered in \cite{Borvert}. The  OPEs of the remaining terms with $J$ are
computed by 
$$m_i^{ferm}(z)
\ee^{\int m^{bos}(z)}J(w)
\sim \frac{(-m_i^{ferm}\ee^{\int m^{bos}(z)}
+m_i^{ferm}\ee^{\int m^{bos}(z)})}{(z-w)}\sim 0.
$$
The OPEs with $L$ are a bit more bothersome. We have 
$$
m_i^{ferm}(z)\ee^{\int m^{bos}(z)}L(w)
\sim (z-w)^{-1} m_i^{ferm}(z) (-m^{bos}(w)\ee^{\int m^{bos}(z)})
$$
$$
+(z-w)^{-1} (-\partial_z m_i^{ferm}\ee^{\int m^{bos}(z)})
+\partial_w\Big((z-w)^{-1} m_i^{ferm}(z) \ee^{\int m^{bos}(z)}\Big)
$$
$$
\sim (z-w)^{-2}m_i^{ferm}(z) \ee^{\int m^{bos}(z)}
+(z-w)^{-1}(-m_i^{ferm}m^{bos}-\partial_zm_i^{ferm})\ee^{\int m^{bos}}
$$
$$
\sim (z-w)^{-2}m_i^{ferm}(w) \ee^{\int m^{bos}(w)}
$$
which shows that the differential acts trivially on the corresponding field.
\end{proof}

\begin{remark}
Given the match of the data, the reader should already find it plausible 
that the algebras $V_{(F^\cdot),g}$ are the algebras of the (0,2)
models considered in \cite{Witten}. 
In what follows we will strengthen their connection to
the (0,2) models by showing that analogous "limit" algebra which is the 
cohomology of ${\rm Fock}_{M\oplus K^\dual}^\Sigma$ with respect to $D_{(F^\cdot),g}$ 
is isomorphic to the cohomology of an analog of the chiral de Rham complex
defined for deformations of the chiral de Rham complex in \cite{gerbes2,gerbes3}.
We closely follow  \cite{Borvert} and overcome the 
fairly minor technical difficulties that occur along the way.
\end{remark}

\section{A cohomology construction of a twisted chiral de Rham sheaf in a particular case}
\label{seccoh}

Let $X$ be a smooth manifold. Let 
$E$ be a vector bundle on $X$ such that $c_1(E)=c_1(TX)$
and $c_2(E)=c_2(TX)$. Assume further that $\Lambda^{\dim X}E$
is isomorphic to $\Lambda^{\dim X}TX$, and, moreover, pick a choice of 
such isomorphism. Then one can construct a collection of sheaves
$\MSV(X,E)$ of vertex algebras on $X$, which differ by regluings given by elements
of $H^1(X,(\Lambda^2 TX^\dual)^{closed})$,
see \cite{gerbes1,gerbes2,gerbes3}.
Locally, any such sheaf is again generated by 
$b^i, a_i, \phi^i,\psi_i$, however $\phi^i$ and $\psi_i$ now transform
as sections of $E^\dual$ and $E$ respectively. 
The OPEs between the $\phi$ and $\psi$ are governed by the pairing
between sections of $E^\dual$ and $E$.
The sheaves $\MSV(X,E)$ carry a natural structure of graded sheaves of vertex algebras.
If, in addition, $X$ is a Calabi-Yau, and one fixes a choice of the nonzero holomorphic volume form, 
then each of the sheaves $\MSV(X,E)$ acquires a conformal structure, as well as an additional affine $U(1)$ current $J(z)$ on it.

\smallskip
The goal of this section is to construct more explicitly 
a twisted chiral de Rham sheaf of $(X,E)$ for 
a particular class of $X$ and $E$. Specifically, if $X$ is a codimension one subvariety 
in a smooth variety $Y$ and $E$ is determined by a global holomorphic one-form on
a line bundle $W$ over $Y$, then we will be able to
calculate $\MSV(X,E)$ in terms of the usual chiral de Rham complex on $W$.

\smallskip
Let $\pi:W\to Y$ be a line bundle over an $n$-dimensional manifold $Y$ with zero section 
$s:Y\to W$. 
Let $\alpha$ be a holomorphic one-form on $W$ which is 
linear with respect to the natural $\CC^*$ action on $W$, i.e. for $\lambda\in\CC^*$ the
following holds
$\lambda^*\alpha=\lambda\alpha$. Consider the locus $X\subset Y$ of points $y$ 
such that $\alpha(s(y))$ as a function on the tangent space $TW_{s(y)}$ is zero on the vertical subspace.

\smallskip
Locally, we have coordinates $(y_1,\ldots,y_n)$ on $Y$. The bundle $W$ is trivialized so that the coordinates
near $s(y)$ are $(y_1,\ldots, y_n, y_{n+1})$. The homogeneity property of $\alpha$ implies that it is given by 
\begin{equation}\label{alpha}
\alpha=\sum_i y_{n+1}P_i(y_1,\ldots,y_n) dy_i + P(y_1,\ldots,y_n)dy_{n+1}.
\end{equation}
In these coordinates 
$X$ is locally given by $P(y_1,\ldots,y_n)=0$. We assume that $X$ is a smooth codimension one submanifold of $Y$. 

\smallskip
Consider the subbundle $E$ of $TY\vert_X$ which is locally defined as the kernel of 
$s^*\alpha$. We will assume 
that it is of corank $1$. In the local description above this means that $P_i$ and $P$ are not simultaneously zero. If $P_i=\partial_i P$ then $E$ is simply $TX$. 
The goal of the rest of 
this section is to show how a twisted chiral de Rham sheaf
$\MSV(X,E)$ can be defined in terms of the usual chiral de Rham complex of $W$.  

\smallskip
The global one-form $\alpha$ on $W$ gives rise to a fermion field $\alpha(z)$ in the chiral de Rham
complex of $W$. Its residue ${\rm Res}_{z=0}\alpha(z)$ gives an endomorphism of 
$\MSV(W)$ and of its pushforward $\pi_*\MSV(W)$ to $Y$.
\begin{theorem}\label{tricky}
The cohomology sheaf of $\pi_*\MSV(W)$ with respect to ${\rm Res}_{z=0}\alpha(z)$ is isomorphic to a twisted chiral de Rham sheaf of $(X,E)$.
\end{theorem}

\begin{remark}
We are working in the holomorphic category, using strong topology. The analogous
statement in Zariski topology will be addressed in Remark \ref{GAGAtricky}.
\end{remark}

\begin{remark}
We identify vector bundles with their sheaves of holomorphic sections.
The weight one component of the pushforward to $Y$ of the 
sheaf of holomorphic $1$-forms on $W$ can be included into
a short exact sequence of locally free sheaves on $Y$
$$
0\to TY^\dual\otimes W^\dual \to (\pi_*TW^\dual)_{1}\to W^\dual\to 0.
$$
The global section $\alpha$ as above induces a global section of $W^\dual$.
Its zero set is precisely $X$. When the above sequence restricts to $X$, 
the section $\alpha\vert_X$ can be identified with a section of $TY^\dual\vert_X\otimes
 W^\dual\vert_X$,
which gives a map $TY\vert_X \to W^\dual\vert_X$. We assume that this map is surjective
and the kernel is the bundle $E$. Thus we have 
$$
0\to E\to TY\vert_X \to W^\dual\vert_X=N(X\subseteq Y)\to 0.
$$
Consequently, $c(E)=c(TX)$ and
the cohomological obstruction of \cite{gerbes2} vanishes.
However (as was pointed to us by Malikov),
it is still rather surprising that one can make a particular choice of the twisted chiral 
de Rham sheaf, distinguished from its possible regluings by elements of the cohomology group
$H^1(X,(\Lambda^2 TY^\dual)^{closed})$. In the case $E=TX$ such a choice 
exists by \cite{gerbes2,gerbes3} but there is no clear explanation for this phenomenon in general.
\end{remark}

The proof of Theorem \ref{tricky} proceeds in several steps. 
First, we calculate the cohomology with respect to 
${\rm Res}_{z=0}\alpha(z)$ for small conformal weights. Then we calculate the OPEs of the 
fields we have found to show that they satisfy the free bosons and free fermions OPEs of 
the twisted chiral de Rham sheaf. This implies that the corresponding Fock space sits inside 
the cohomology. Then we calculate the new $L$ and $J$ fields in terms of these 
free fields. Finally, we use induction on the sum of the conformal weight and the 
fermion number to show that the cohomology
algebra contains no additional fields.

\smallskip
We work in local coordinates as in \eqref{alpha}. We have the fields $\phi^i$,
$\psi_i$, $a_i$ as well as the fields $b^i$ that correspond to the variables $y_i$.
In these coordinates, we have
$$
{\rm Res}_{z=0}\alpha(z)={\rm Res}_{z=0}\Big(
\sum_{i=1}^n b^{n+1}(z)P_i({\bf b}(z))\phi^i(z)
+P({\bf b}(z))\phi^{n+1}(z)
\Big).$$
Observe that ${\rm Res}_{z=0}\alpha(z)$ has conformal weight $(-1)$ and fermion number $1$. There is also an additional integer grading by the acton of $\CC^*$ and this differential has weight $1$ with respect to it. Let us calculate the cohomology for small conformal weights. 

\begin{lemma}\label{hatb}
The cohomology sheaf of $\pi_*\MSV(W)$ with respect to ${\rm Res}_{z=0}\alpha(z)$
is supported on $X$. The conformal weight zero and fermion number zero subsheaf
is isomorphic to ${\mathcal O}_X$.
\end{lemma}

\begin{proof}
For an open set $U_Y\subseteq Y$, 
the subsheaf of $\pi_*{\rm MSV}(W)$ of conformal weight zero and fermion number zero 
is the sheaf of holomorphic functions on $\pi^{-1}U_Y$. The fields that can map
to it under ${\rm Res}_{z=0}\alpha(z)$
 are of conformal weight one and fermion number $(-1)$. These are linear combinations of 
$\psi$-s with 
coefficients that are functions in $b$-s, which correspond to the vector fields on 
$\pi^{-1}U_Y$.
The result of applying ${\rm Res}_{z=0}\alpha(z)$ amounts to pairing of $\alpha$ with that vector field. Thus the image is the ideal generated by $y_{n+1}P_i$ and $P$. Since $P$ and $P_i$ have no common zeroes,
this is the same as the ideal generated by $y_{n+1}$ and $P$. The quotient is then naturally isomorphic to
the sheaf of holomorphic functions on $U_X=U_Y\cap X$. 
\end{proof}

\begin{lemma}\label{hatpsi}
The conformal weight one and the fermion number $(-1)$ cohomology sheaf 
of $\pi_*\MSV(W)$ with respect to ${\rm Res}_{z=0}\alpha(z)$
is naturally isomorphic to the sheaf of sections of $E$.
\end{lemma}

\begin{proof}
In the notations of the proof of Lemma \ref{hatb}, the 
conformal weight one and fermion number $(-1)$ subspace of 
$\pi_*\MSV(W)(U_Y)$  is the space of vector 
fields on $\pi^{-1}U_Y$. The kernel of the map consists of all fields which contract to 
$0$ by $\alpha$. These fields are given locally by $\sum_{i=1}^{n} Q_i\partial_i+Q\partial_{n+1}$ with 
\begin{equation}\label{1}\sum_i y_{n+1}P_iQ_i + QP =0.
\end{equation}
Here $Q_i,Q$ are functions of $(y_1,\ldots,y_{n+1})$. Observe that $Q$ is necessarily 
divisible by $y_{n+1}$, so we have $Q=y_{n+1}\tilde Q$ and 
$$
\sum_i P_iQ_i + \tilde Q P =0.
$$
Since the functions $P_i$ and $P$ have no common zeroes, the corresponding 
Koszul complex is acyclic, and the solutions to the above equation are generated,
as a module over the functions on $\pi^{-1}U_Y$, 
by $(Q_i=P,\tilde Q=-P_i)$, which corresponds
to $P\partial_i - y_{n+1}P_i\partial_{n+1}$ and $(Q_i = P_j,Q_j=P_i)$ which correspond to 
$P_j\partial_i-P_i\partial_j$ for $1\leq i,j\leq n$.

\smallskip
We need to take a quotient of this space by the image of the space of conformal weight two and fermion number $(-2)$. These are made from second exterior powers of the tangent bundle. The action is the contraction by $\alpha$. Consequently, the image is the submodule generated by 
\begin{equation}\label{2}
y_{n+1}P_i \partial_{n+1} - P\partial_i,~~y_{n+1}P_i\partial_j - y_{n+1}P_j\partial_i.
\end{equation}
Let us consider the quotient module. By using $y_{n+1}P_i \partial_{n+1} - P\partial_i$
we can reduce the quotient to the quotient of the module spanned by 
$P_j\partial_i-P_i\partial_j$ by the fields spanned by $y_{n+1}P_i\partial_j - y_{n+1}
P_j\partial_i$ as well as any linear combinations of the fields $y_{n+1}P_i\partial_{n+1}- P\partial_i$ which 
have no $\partial_{n+1}$. These are precisely the terms of the form
 $P$ times any linear combination of 
$P_j\partial_i-P_i\partial_j$. This means that we are taking the quotient of the space of
sections of the tangent bundle on $Y$ that satisfy $\sum_i Q_i P_i = 0$ by
$P$ times these sections. We observe that the result is precisely the sections
of the vector bundle $E$ on $U_X=X\cap U_Y$.
\end{proof}

Similarly we can handle the conformal weight zero and fermion number $1$ case.
\begin{lemma}\label{hatphi}
The conformal weight one and the fermion number $1$ cohomology sheaf 
of $\pi_*\MSV(W)$ with respect to ${\rm Res}_{z=0}\alpha(z)$
is naturally isomorphic to the sheaf of sections of $E^\dual$.
\end{lemma}

\begin{proof}
For conformal weight zero and fermion number $1$, we are looking at the quotient of the sheaf of 
differential one forms on $W$ by the image of ${\rm Res}_{z=0}\alpha(z)$ of the sheaf of fields of 
the form $f_i^j(b)\phi^i\psi_j + g^i(b)a_i$. The quotient by the image of the first kind of fields
is simply the fields of the restriction of $TW$ to $X$. Indeed, these are simply obtained by multiplying
the cokernel of the differential at conformal weight zero and fermion number zero by $\phi^i$.

\smallskip
For $1\leq j\leq n$ we have 
$${\rm Res}_{z=0}\alpha(z)a_j = \sum_{i}b^{n+1}\partial_j P_i \phi^i + 
\partial_j P \phi^{n+1}.$$
Since $b^{n+1}$ is trivial in the cohomology, we can reduce this to $\partial_j P \phi^{n+1}$. 
Since the functions $\partial_j P$ 
have no common zeroes (because $X$ is smooth), we see that $\phi^{n+1}$ lies in the image and 
is trivial in cohomology. It remains to take the quotient by the module
generated by  ${\rm Res}_{z=0}\alpha(z)a_{n+1}$. We have 
$$
{\rm Res}_{z=0}\alpha(z)a_{n+1}=\sum_{i=1}^n P_i\phi^i.
$$
Thus we see that the cohomology fields of this conformal weight and fermion number are 
naturally isomorphic to the sections of the dual bundle of $E$.
\end{proof}

\begin{remark}
The above calculations give us the fields that will correspond to the free fermions of 
$\MSV(X,E)$. Let us calculate their OPEs. Clearly, OPEs of the fields from Lemma
\ref{hatphi} 
with each other are trivial and similarly for the OPEs of the fields from Lemma \ref{hatpsi}.
Let us 
calculate the OPE of a field from Lemma \ref{hatphi} with a field from Lemma \ref{hatpsi}. 
We can take an old $\phi^j$ 
to be a representative of a field from Lemma \ref{hatphi}. Then its OPE with 
$\sum_i P_i \psi_i$ is 
$$
\sim \frac 1{z-w} P_i.
$$
Since the pairing between $E$ and $E^\dual$ is induced from pairing between
$TX$ and $TX^\dual$, we see that our new fields have the pairings
expected for the fields of $\MSV(X,E)$.
\end{remark}

Assume for a moment that $P_n\neq 0$ and $P=y_n$. Then the following fields 
will provide the generators of the cohomology with respect to
${\rm Res}_{z=0}\alpha(z)$. We will show that they always generate the cohomology
a bit later, in Lemma \ref{sub}. For now we will just study their OPEs.
\begin{definition}\label{hatfields}
For $1\leq j\leq n-1$ consider
$$
\begin{array}{c}
\hat b^j :=b^j,~\hat \phi^j:=\phi^j,~
\hat \psi_j:=\psi_j-P_jP_n^{-1}\psi_n,~
\\[.5em]
\hat a_j:=a_j-\sum_{i=1}^n (\partial_jP_i)P_n^{-1}\phi^i\psi_n
-\frac 12P_n^{-2}\partial_j P_n (P_n)'.
\end{array}
$$
Here $(P_n)'=\partial_zP_n$ refers to the differentiation with respect to the variable
on the world-sheet. Also in the $i=n$ term for the summation for $\hat a_j$
we implicitly assume normal ordering.
\end{definition}

\begin{lemma}\label{hatOPE}
The fields $\hat b^j,\hat a_j, \hat \phi^j,\hat \psi_j$
lie in the kernel of ${\rm Res}_{z=0}\alpha(z)$ and thus descend
to the cohomology.
We have 
$$
\hat a_j(z)\hat b^k(w)\sim \frac {\delta_{j}^k}{z-w},~~
\hat \phi^k(z)\hat \psi_j(w)\sim \frac {\delta_{j}^k}{z-w}
$$
with other OPEs nonsingular.
\end{lemma}

\begin{proof} 
These are routine calculations using Wick's theorem for OPEs of products
of free fields.
It is important to use $\partial_j P= \partial_j y_n=0$
for $j$ in the above range. 

\smallskip
We will do the more tricky of the calculations and leave the rest to the reader.
For example, let us calculate the OPE of $\hat a_j(z)$ and $\alpha(w)$.
We have 
$$
\begin{array}{rl}
\hat a_j(z)\alpha(w)\sim& \Big(a_j(z)-\sum_{i=1}^n (\partial_jP_i)P_n^{-1}\phi^i(z)\psi_n(z)\Big)
\\&\Big(\sum_{k=1}^nb^{n+1}(w)P_k(w)\phi^k(w)+b^n(w)\phi^{n+1}(w)\Big)
\\
\sim &(z-w)^{-1}\Big(
\sum_{k=1}^n b^{n+1}\partial_jP_k\phi^k
\\&
-\sum_{i,k=1}^n(\partial_jP_i)P_n^{-1}b^{n+1}P_k\phi^i\delta_{n}^k
\Big)
\sim 0.
\end{array}
$$
The OPE of $\hat a$ and $\hat \psi$ fields is computed as follows.
$$
\begin{array}{rl}
\hat a_j(z)\hat \psi_k(w) \sim&
\Big(a_j(z)-\sum_{i=1}^n (\partial_jP_i)(z)P_n^{-1}(z)\phi^i(z)\psi_n(z)\Big)\\
&
\Big(\psi_k(w)-P_k(w)P_n^{-1}(w)\psi_n(w)\Big)
\\
\sim &
(z-w)^{-1}\Big(-\partial_j(P_kP_n^{-1})\psi_n+
(\partial_jP_k)P_n^{-1}\psi_n
\\ &-P_n^{-1}(\partial_jP_n)P_kP_n^{-1}\psi_n\Big)
\sim 0.\end{array}
$$
In the above calculations we ignored the dependence of the terms 
of the coefficient at $(z-w)^{-1}$ on $z$ versus $w$, since the difference 
is nonsingular.

\smallskip
By far the most complicated calculation is
the OPE of $\hat a_j(z)\hat a_k(w)$.
This OPE has poles of order two at $z=w$. We need to be careful
with the second order terms to include the dependence on the variables.
The coefficient at $(z-w)^{-2}$ is coming from the double pairings of the $\phi^n\psi_n$ terms
and the pairing between the $a$-s and the $(P_n)'$ terms.
It is given by
\begin{equation}\label{order2}
\begin{array}{c}
(\partial_j P_n)(z)P_n^{-1}(z)(\partial_k P_n)(w)P_n^{-1}(w)\\[1em]
-\frac 12 P_n^{-2}(w)(\partial_k P_n)(w)(\partial_jP_n(w)) 
-\frac 12 P_n^{-2}(z)(\partial_jP_n(z))(\partial_kP_n(z))
\end{array}
\end{equation}
The above expression is zero at $z=w$. However, these pairings contribute to the 
coefficient by $(z-w)^{-1}$. Specifically, \eqref{order2} contributes
\begin{equation}\label{blah1}
\begin{array}{c}
\partial_w ((\partial_j P_n)P_n^{-1}) (\partial_k P_n)P_n^{-1}
-\frac 12\partial_w(P_n^{-2}(\partial_jP_n)(\partial_kP_n))\\[.5em]
=\frac 12P_n^{-2} (\partial_k P_n)(\partial_jP_n)'
-\frac 12 P_n^{-2}(\partial_jP_n)(\partial_kP_n)'.
\end{array}
\end{equation}
Note that the pairing between $a_j(z)$ and $-\frac 12P_n^{-2}(w)\partial_k 
P_n(w) (P_n)'(w)$
additionally contributes to $(z-w)^{-1}$ term as follows. We have
\begin{equation}\label{blah2}
\begin{array}{c}
-\frac 12P_n^{-2}(w)(\partial_k P_n)(w) \partial_w((z-w)^{-1}\partial_jP_n(w))\sim
\\[.5em]
\hskip-5pt
-\frac 12P_n^{-2}(w)(\partial_kP_n)(w)\partial_jP_n(w)(z-w)^{-2}
\hskip -3.5pt
-\frac 12P_n^{-2}(\partial_kP_n)(\partial_jP_n)'(z-w)^{-1}
\end{array}
\end{equation}
of which only the first term was accounted for in \eqref{order2}.
Similarly, the OPE of $-\frac 12P_n^{-2}(z)\partial_j P_n(z)  (P_n)'(z)$
and $a_k(w)$ will yield
\begin{equation}\label{blah3}
\begin{array}{c}
\frac12 P_n^{-2} (z) (\partial_j P_n)(z)\partial_z((z-w)^{-1}(\partial_kP_n)(z))
\sim
\\[.5em]
\hskip -6pt
-\frac12 P_n^{-2}(z)(\partial_jP_n)(z)(\partial_kP_n)(z)(z-w)^{-2}
\hskip -3.5pt
+\frac12 P_n^{-2}(\partial_jP_n)(\partial_kP_n)'(z-w)^{-1}
\end{array}
\end{equation}
of which the second term is not accounted for in \eqref{order2}.
Note that the second terms of \eqref{blah2} and \eqref{blah3} cancel 
the contribution of \eqref{blah1}.

\smallskip
There are additional contributions to the $(z-w)^{-1}$ term of the OPE
that come from other pairings in the Wick's theorem. 
We need
to consider the pairings of $a_i$ and $a_k$ with the 
functions of $b$-s. We also need to consider the results of pairings
of $\phi^n\psi_n$ terms with $\phi^i\psi_n$ terms.
The coefficient at $(z-w)^{-1}$ is then calculated to be
\begin{equation}
\begin{array}{c}
\sum_{i=1}^n \partial_k((\partial_jP_i)P_n^{-1})\phi^i\psi_n
-\sum_{i=1}^n\partial_j((\partial_kP_i)P_n^{-1})\phi^i\psi_n
\\[.5em]
-\sum_{i=1}^n (\partial_jP_n)P_n^{-2} \partial_k P_i\phi^i\psi_n
+\sum_{i=1}^n (\partial_kP_n)P_n^{-2} \partial_j P_i\phi^i\psi_n
\\[.5em]
+\frac 12\partial_k(P_n^{-2}\partial_jP_n)(P_n)'
-\frac 12\partial_j(P_n^{-2}\partial_kP_n)(P_n)'
=0.\end{array}
\end{equation}
\end{proof}

In the next lemma we will calculate a Virasoro and the $U(1)$ current
fields for the fields of Definition \ref{hatfields}.
\begin{definition}\label{defhatLJ}
Define 
$$
\hat J:= \sum_{j=1}^{n-1}\hat \phi^j\hat\psi_j,~~\hat L : = \sum_{j=1}^{n-1}
(\hat b^j)'\hat a_j + \sum_{j=1}^{n-1}(\hat \phi^j)'\hat \psi_j.
$$
where we are using the normal ordering from the modes of the free $~\hat{}~$ fields.
\end{definition}

\begin{lemma}\label{hatLJ}
We have the following equalities in the cohomology of $\MSV(U_W)$ 
by ${\rm Res}_{z=0}\alpha(z)$
$$
\hat J=\sum_{j=1}^{n+1} \phi^j\psi_j  
-\Big(b^{n+1}a_{n+1}+ \phi^{n+1}\psi_{n+1}\Big)-(\ln P_n)'
$$
$$
\hat L = \sum_{j=1}^{n+1}(b^j)' a_j + \sum_{j=1}^{n+1} (\phi^j)'\psi_j
+\frac 12 (\ln P_n)''-\Big(b^{n+1}a_{n+1}+ \phi^{n+1}\psi_{n+1}\Big)'
$$
where on the right hand side we are using the normal ordering with respect 
to the free fields on $\pi_*\MSV(W)$.
\end{lemma}

\begin{proof}
Let us first try to calculate $\hat J$. To calculate the normal ordered products, 
we subtract the singular terms of the OPEs to get 
$$
\hat J = \sum_{j=1}^{n-1} \hat \phi^j\hat \psi_j =
\sum_{j=1}^{n-1} \phi^j\psi_j  - \sum_{j=1}^{n-1}P_jP_n^{-1}\phi^j\psi_n
$$
$$=\sum_{j=1}^{n+1} \phi^j\psi_j - \phi^{n+1}\psi_{n+1} -  \sum_{j=1}^{n}P_jP_n^{-1}\phi^j\psi_n.
$$
Consider 
$$
\begin{array}{c}
\alpha(z)P_n^{-1}(w)a_{n+1}(w)\psi_n(w)
\\[.5em]
=(\displaystyle\sum_{i=1}^n b^{n+1}(z)P_i(z)\phi^i(z)+b^n(z)\phi^{n+1}(z))
P_n^{-1}(w)a_{n+1}(w)\psi_n(w)
\\[.5em]
\sim
-(z-w)^{-2}P_n(z)P_n^{-1}(w)
+(z-w)^{-1}(b^{n+1}a_{n+1}-\displaystyle\sum_{i=1}^n P_i\phi^iP_n^{-1}\psi_n)
\\[.5em]
\sim -(z-w)^{-2} +(z-w)^{-1}
(b^{n+1}a_{n+1}-\displaystyle\sum_{i=1}^n P_i\phi^iP_n^{-1}\psi_n-P_n^{-1}P_n').
\end{array}
$$
Thus,
$${\rm Res}_{z=0}\alpha(z)(P_n^{-1}a_{n+1}\psi_n)
=b^{n+1}a_{n+1}-\sum_{j=1}^n P_jP_n^{-1}\phi^j\psi_n-P_n^{-1}P_n',
$$
so the field $\hat J$ 
is equivalent to $\sum_{j=1}^{n+1} \phi^j\psi_j - \phi^{n+1}\psi_{n+1}-b^{n+1}a_{n+1}
-(\ln P_n)'$.

\smallskip
The calculation for $\hat L$ is similar though more complicated. 
The difference between it and the right hand side of Lemma \ref{hatLJ}
turns out to equal the image under ${\rm Res}_{z=0}\alpha(z)$ of the 
field
$$
P_n^{-1}\psi_na_{n+1}' + \sum_{j=1}^n(\partial_n P_j)P_n^{-1}\phi^j\psi_n\psi_{n+1}'
-\psi_{n+1}'a_n.
$$
Details are left to the reader.
\end{proof}

In the following lemma, we will show that the free fields $\hat b^j,\hat \phi^j,
\hat \psi_j,\hat a_j$ locally generate the cohomology of ${\rm MSV}(W)$ by 
${\rm Res}_{z=0}\alpha(z)$. 

\begin{lemma}\label{sub}
Let $x\in X$ be a point. Pick a small open subset $U_X\subset X$ containing $x$.
We can pick coordinates on $Y$ such that $y_n=P$. By changing $y_i$ to $y_i+y_n$ 
and possibly shrinking $U_X$ we can also assume that $P_n \neq 0$ on $U_X$. 
Pick $U_Y$ to be an open subset on $Y$ with $U_Y\cap X =U_X$
and denote by $U_W$ the preimage of $U_Y$ in $W$. Then the cohomology
of ${\rm MSV}(U_W)$ with respect to ${\rm Res}_{z=0}\alpha(z)$ is generated by 
the $4(n-1)$  free fields $\hat b^j,\hat \phi^j,
\hat \psi_j,\hat a_j$, for $1\leq j\leq n-1$.
\end{lemma}

\begin{proof}
From Lemma \ref{hatOPE} we see that the above fields generate a subalgebra of 
the cohomology. Since we have a description of the cohomology of the conformal
weight zero piece as the functions on $U_X$, we see that their OPEs imply that 
this subalgebra is the usual Fock space representation, namely polynomials 
in negative modes of $\hat b$, $\hat a$, $\hat \psi$ and nonpositive modes of 
$\hat\phi$, tensored with functions on $U_X$ for the zero modes of $\hat  b$.

\smallskip
Let us show that there are no additional cohomology elements. We will first 
handle the part of the cohomology where the fermion number plus the 
conformal weight of $\pi_* \MSV(W)$ is zero. This is the cohomology of 
the algebra of polyvector fields on $U_W$ with respect to the contraction by $\alpha$.
We have already seen this at fermion number $(-1)$ in Lemma \ref{hatpsi}.
This is a Koszul complex for the ring ${\mathcal O}(U_W)$
and functions $y_{n+1}P_i$ and $y_n$. We can think of it as an exterior algebra
over the ring of ${\mathcal O}(U_W)$ of the vector space with the basis
$\hat\psi_j, 1\leq j\leq n-1$, $\psi_n$, $\psi_{n+1}$. Then we have the Koszul 
complex for $y_{n+1}P$ and $y_n$ for the ring ${\mathcal O}(U_W)$ tensored
with the exterior algebra in $\hat\psi_j$. It remains to observe that this 
Koszul complex has cohomology only at the degree zero term which 
is equal to ${\mathcal O}(U_X)$.

\smallskip
We will proceed by induction on the conformal weight plus fermion number.
Conformal weight plus fermion number is simply the eigenvalue of the operator $H$ 
which is the coefficient of
$z^{-2}$ of $L(z)-J(z)'$. By Lemma \ref{hatLJ}, this operator $H$ is equal to the $z^{-2}$
coefficient 
of $\hat L(z)-\hat J(z)'$. We can write 
$$
\hat b^j(z)=\sum_{n\in\ZZ}\hat b^j[n]z^{-n},~
\hat a_j(z)=\sum_{n\in\ZZ}\hat a_j[n]z^{-n-1},~
$$$$
\hat \phi^j(z)=\sum_{n\in\ZZ}\hat \phi^j[n]z^{-n-1},~
\hat \psi_j(z)=\sum_{n\in\ZZ}\hat \psi_j[n]z^{-n},~
$$
where the endomorphisms with index $[n]$ change the $H$-degree
of homogeneous elements   by $(-n)$.
We have 
\begin{equation}\label{no}
\begin{array}{rl}
H=& \sum_{n\in \ZZ_{>0}} \sum_j (-n)\hat a_j[-n]\hat b^j[n] + \sum_{n\in\ZZ_{<0}} \sum_j 
(-n)\hat b^j[n]\hat a_j[-n]
\\[.5em]
&+\sum_{n\in\ZZ_{>0}} \sum_j (-n)\hat \phi^j[-n]\hat \psi_j[n]
-\sum_{n\in\ZZ_{<0}} \sum_j (-n)\hat \psi_j[n]\hat \phi^j[n].
\end{array}
\end{equation}

\smallskip 
Suppose we have proved the statement of the lemma
for all eigenvalues of $H$ that are less than some positive integer $r$.
If an element $v$ of the cohomology of $\MSV(U_W)$ with respect
to ${\rm Res}_{z=0}\alpha(z)$ has positive $H$-eigenvalue $r$,
then we have $v=\frac 1r Hv$.
Because of the normal ordering, when calculating $Hv$ as in \eqref{no} one is applying first
the modes that decrease the eigenvalue of $H$ and thus by induction
send $v$ into the subalgebra generated by the $\hat{}$ fields. Thus $Hv$ lies in this algebra,
which furnishes the induction step. 
\end{proof}

\begin{remark}
While the cohomology of $\pi_*\MSV(W)$ with respect to ${\rm Res}_{z=0}\alpha(z)$
is a well-defined sheaf of vertex algebras, the conformal structure is a priori not 
clear. In general, to define a conformal structure for the twisted chiral de Rham sheaf 
for a vector bundle $E$ one needs to choose an isomorphism between $\Lambda^{n-1}E$
and $\Lambda^{n-1}TX$ (up to constant multiple). Specifically, one needs to be 
sure that in the local coordinates the exterior product of $\hat \phi^j$ corresponds
to the exterior product of $d\hat b^j$ under the dual of the above isomorphism.
There is a natural choice of isomorphism here that works globally for $X$ as follows.
We can think of the restriction $\alpha\vert_X$ as a section of 
$W^\dual\otimes TY\vert_X^\dual$, or a map $TY\vert_X\to W\vert X$.
Then it defines a short exact sequence of
bundles on $X$
$$
0\to E\to TY\vert_X\to W\vert X\to 0.
$$
This provides a natural identification of  $\Lambda^{n-1}E$
and $\Lambda^{n-1}TX$. Locally this amounts to the multiplication 
by $P_n$. In the notations above $\hat\phi$ and $d\hat b$ are not compatible,
which accounts for the presence of the  extra term $\frac 12(\ln P_n)''$ in $\hat L$.
Consequently, for the globally defined conformal structure on the cohomology,
we need to use $\hat L$ and $\hat J$ that are defined for $P_n$ that's constant
on $X$, in which case the extra terms in Lemma \ref{hatLJ} do not appear.
\end{remark}

We are now ready to prove Theorem \ref{tricky}.
\begin{proof}
By Lemma \ref{sub} the cohomology is locally isomorphic to a free field vertex algebra.
The conformal weight zero and fermion one and weight one and fermion 
number $(-1)$  parts are naturally isomorphic to the sheaves 
of sections of $E^\dual$  and $E$ by Lemmas \ref{hatphi} and \ref{hatpsi} respectively.
The statement now follows from \cite{gerbes2}.
\end{proof}

\begin{remark}\label{beta}
Let us examine in more detail the field 
$$
\beta =b^{n+1}a_{n+1}+ \phi^{n+1}\psi_{n+1}
$$
featured prominently in Lemma \ref{hatLJ}.
The action of $\CC^*$ on $W$ canonically defines a vector field which 
in local coordinates looks like $\psi_{\CC^*}=b^{n+1}\psi_{n+1}$. Consider the OPE 
of the field 
$$
Q(z) = \sum_{i=1}^{n+1}a_{i}\phi^i 
$$
with $\psi_{\CC^*}$. 
We get
$$
Q(z)\beta(w) \sim (z-w)^{-2} + (z-w)^{-1}\beta(w).
$$
Consequently, $\beta$ is the image of $\psi_{\CC^*}$ under the 
map ${\rm Res}_{z=0}Q(z)$. While $Q$ itself depends on the choice 
of coordinates, see \cite[equation 4.1(c)]{MSV}, its residue does not.
Thus, $\beta$ is independent of the choice of the coordinate system.
\end{remark}

\begin{remark}
If in addition the bundle $W$ is the canonical line bundle on $Y$, then the total 
space of $W$ is a Calabi-Yau. It has a natural nondegenerate volume form
which is the derivative of the image in $\Lambda^{n}TW^\dual$
of the tautological section of $\pi^* \Lambda^{n}TY^\dual$.
Thus, the $J$ field on $W$ is well-defined
as is the field $\hat J$ on $X$. In fact, Lemma \ref{hatLJ} shows that $\hat J$
is the image in the cohomology of the field 
$$
J - \beta
$$
where $\beta$ is defined in the above remark.
The field $J-\beta$
descends naturally to $X$, which in this case is also a Calabi-Yau. The particular
case when $\alpha$ was a gradient of a global function, linear on 
fibers, was considered in \cite{Borvert}.
In this case, we get the usual (not twisted) chiral de Rham complex 
on $X$, with $N=2$ structure. 
\end{remark}

\begin{remark}\label{GAGAtricky}
We observe that Theorem \ref{tricky} holds in the algebraic setting.
Namely, if $Y$, $W$, $X$ and $\alpha$ are algebraic, then the statement holds
for sheaves of vertex algebras in Zariski topology. Indeed, the calculations of 
Lemmas \ref{hatb}-\ref{hatphi} are unchanged. We can pick rational functions $y_i$
and Zariski open subsets $U_X$, $U_Y$ and $U_W$ as before, so that they
generate the $m/m^2$ at all points in $U_W$. Then the partial derivatives of
rational functions make sense as rational functions and the calculations of 
Lemmas \ref{hatLJ} and \ref{sub} and Theorem \ref{tricky} go through as well.
\end{remark}

\begin{proposition}\label{anyaff}
For \emph{any affine} Zariski open subset $U_Y$ the cohomology 
of $\pi_*\MSV(W)$ on $U_Y$ by ${\rm Res}_{z=0}\alpha(z)$ 
is isomorphic to the sections of $\MSV(X,E)$ on $U_X=U_Y\cap X$.
\end{proposition}

\begin{proof}
We can cover $U_Y$ by smaller subsets on which 
the statement holds. Then the statement  holds on their intersections
by localization. The \v{C}ech complexes for $\pi_*\MSV(W)$ and 
$\MSV(X,E)$ for this cover of $U_Y$ have no higher cohomology, because 
these sheaves are filtered with quasi-coherent quotients. Then 
the snake lemma finishes the proof.
\end{proof}

\begin{remark}
It appears plausible that one can replace the line bundle $W$ by a vector bundle and
apply the calculations of this section to subvarieties $X\subseteq Y$ which are defined by sections of a vector bundle. In particular, the approach should work for complete intersections of hypersurfaces.
\end{remark}

\begin{remark}
It would be interesting to study to what extent one can use this approach to \emph{define}
the (twisted) chiral de Rham sheaf for hypersurfaces with some mild singularities.
\end{remark}

\section{Deformations of the cohomology of twisted chiral de Rham sheaf and CY/LG correspondence}\label{CYLG}

In this section we want to show that the vertex algebras $V_{(F^\cdot),g}$ 
of Definition \ref{hetd} are 
in some sense deformations of the cohomology of a twisted chiral de Rham sheaf
constructed in \cite{gerbes1,gerbes2,gerbes3} and further studied in \cite{Tan}. Specifically, we will show that the cohomology of the chiral de Rham sheaf
for the vector bundle on the quintic considered in \cite{Witten} is equal to the cohomology 
of ${\rm Fock}_{M\oplus K^\dual}^\Sigma$ by the operator $D_{(F^\cdot),g}$ 
defined in Section \ref{sechet}. Our method also shows how one can produce 
more examples of calculations of cohomology of twisted chiral de Rham sheaf on 
hypersurfaces and complete intersections.

\smallskip
Let $x_i,0\leq i\leq 4$ be homogeneous coordinates in $\PP^4$.
Let $F^i=x_iR^i, 0\leq i\leq 4$ be homogeneous polynomials of degree $5$
as in Section \ref{sechet}.  Consider the lattice vertex algebra ${\rm Fock}_{M\oplus K^\dual}^\Sigma$ and the operator
$$
D_{(F^\cdot),g}={\rm Res}_{z=0}\Big(\sum_{\stackrel{m\in \Delta}{0\leq i\leq 4}}F^i_m m_i^{ferm}(z)\ee^{\int m^{bos}(z)}
+\sum_{n\in\Delta^\dual} g_n n^{ferm}(z)\ee^{\int  n^{bos}(z)}\Big)
$$ 
from Section \ref{sechet}. 
\begin{theorem}\label{5.1}
The cohomology of  ${\rm Fock}_{M\oplus K^\dual}^\Sigma$ with 
respect to $D_{(F^\cdot),g}$ is isomorphic to the cohomology of a twisted chiral de Rham 
sheaf on the quintic $\sum_{i=0}^4F^i=0$ given by $R^i$.
\end{theorem}

\begin{proof}
Consider the canonical bundle $\pi:W\to \PP^4$.
Over the chart $x_j\neq 0$ on $\PP^n$ the coordinates on $W$ are 
$\frac {x_i}{x_j},i\neq j$ and $s_j$. The coordinate changes 
are  $s_k=s_j\Big(\frac {x_k}{x_j}\Big)^5$.
The data $(F^i=x_iR^i)$ give rise to a  $1$-form in an affine chart $x_k\neq 0$
defined as 
$$
\alpha_k = x_k^{-4}s_k\sum_{i=0}^4 R^i(x) d\Big(\frac {x_i}{x_k}\Big)   
+ \frac 15x_k^{-5}\sum_{i=0}^4 x_iR^i ds_k.
$$
It is easily checked that these forms glue together to a global $1$-form $\alpha$
on $W$ which is of weight one with respect to the $\CC^*$ action on the fibers.
The vector bundle $E$ on $X=\{\sum_i F^i=0\}$ constructed 
from this form in Section \ref{seccoh} is isomorphic to the bundle considered in
\cite{Witten}.

\smallskip
All further arguments are essentially identical to those of \cite{Borvert}. 
One considers the cover of $\PP^4$ and its canonical bundle $W$ 
by toric affine charts. The cone $K^\dual$ is subdivided by a fan $\Sigma$.
The cones of this fan correspond to toric charts on $W$.
It was already seen in \cite{Borvert}
that for a chart that corresponds to a face $\sigma$
of $K^\dual$, the sections of the chiral de Rham complex on $W$ 
correspond to the cohomology of ${\rm Fock}_{M\oplus\sigma}$ with respect
to
$$
D_g= {\rm Res}_{z=0}\sum_{n\in \Delta^\dual\cap \sigma} g_n n^{ferm}(z)\ee^{\int n^{bos}(z)}.
$$
By Proposition  \ref{anyaff}
the cohomology of a twisted chiral de Rham sheaf $\MSV(X,E)$
of $X=\{\sum_i x_iR^i=0\}$ is isomorphic to the cohomology of 
${\rm Fock}_{M\oplus\sigma}/D_g$ by the ${\rm Res}_{z=0}\alpha(z)$.
It is a routine calculation to check that this corresponds precisely to the 
cohomology via 
$${\rm Res}_{z=0}\sum_{i=0}^4 
\sum_{m\in \Delta}F^i_m m_i^{ferm}(z)\ee^{\int m^{bos}(z)}.$$
The spectral sequence for the cohomology of the sum
degenerates, as in \cite[Proposition 7.11]{Borvert}. This shows that the sections
of the twisted chiral de Rham sheaf over the open chart are isomorphic to 
the cohomology of 
${\rm Fock}_{M\oplus\sigma}$ by $D_{(F^\cdot),g}$.

\smallskip
Toric \v{C}ech cohomology as in \cite[Theorem 7.14]{Borvert} finishes the proof.
\end{proof}

\begin{remark}
We observe that the fields $J$ and $L$ defined in Proposition \ref{3.5}
correspond precisely to the fields $J$ and $L$ of the twisted chiral de Rham complex.
This is simply 
a matter of going through the calculations. The field $\beta$ of Remark \ref{beta}
turns out to be $(\deg^\dual)^{bos}$. The $\deg^{bos}$ part in Proposition
\ref{3.5} comes from the 
description of chiral de Rham complex in logarithmic coordinates, see \cite[Proposition 6.4]{Borvert}.
\end{remark}

We are now ready to remark on Calabi-Yau/Landau-Ginzburg correspondence
for (0,2) theories. Consider the vertex algebras which are the cohomology
of ${\rm Fock}_{M\oplus K^\dual}$ by $D_{(F^\cdot),g}$ as $(F^\cdot)$ is fixed
and $g$ varies. If we fix $g_n$ for $n\neq \deg^\dual$ and let $g_{\deg^\dual}$
go to $\infty$, then in the limit the action of $D_{(F^\cdot),g}$ starts to
resemble its action on  ${\rm Fock}_{M\oplus K^\dual}^\Sigma$, after an appropriate
reparametrization. This is the Calabi-Yau limit of the theory. The Landau-Ginzburg 
limit occurs for $g_{\deg^\dual}=0$, as in the $N=2$ case.

\begin{remark}
We do not know whether passing from the cohomology of 
${\rm Fock}_{M\oplus K^\dual}^\Sigma$ by $D_{(F^\cdot),g}$
to the cohomology of 
${\rm Fock}_{M\oplus K^\dual}$ by $D_{(F^\cdot),g}$ does not
change the dimension of the graded pieces of the cohomology.
From the physical point of view it is conceivable that instanton corrections
result in some reduction of the dimension of the state space of the 
half-twisted (0,2) theory.
\end{remark}

\section{Chiral rings}\label{chiral}
In this section we discuss the consequences of the machinery of 
\cite{Borvert,chiralrings} as it applies to the algebras $V_{(F^\cdot),g}$.

\smallskip
First, we observe that we can replace the cone $K^\dual$ by the whole
lattice $N$.
\begin{proposition}
The algebra $V_{(F^\cdot),g}$ can be alternatively described as the 
cohomology of ${\rm Fock}_{M\oplus N}$ or ${\rm Fock}_{K\oplus N}$ 
by $D_{(F^\cdot),g}$.
\end{proposition}

\begin{proof}
This statement for the usual algebras $V_{f,g}$ is called the Key Lemma
in \cite{Borvert} because of its importance to mirror symmetry.
The argument is unchanged after one replaces the Koszul complex
for $\CC[K]$ and logarithmic derivatives of $F$ by the Koszul complex
for $\CC[K]$ and $F_i$.
\end{proof}

As in the $N=2$ case we define operators $H_A$ and $H_B$ by 
$$
H_A={\rm Res}_{z=0} zL(z),~~ H_B={\rm Res}_{z=0}(zL(z)+J(z)).
$$
We then define the chiral rings of the theory as the parts of the 
vertex algebra where $H_A=0$ or $H_B=0$. The calculations
of the paper \cite{chiralrings} apply directly to this more general setting.
Consider the commutative ring $\CC[K \oplus K^\dual]$. Consider 
the quotient $\CC[(K \oplus K^\dual)_0]$ by the ideal spanned with 
monomials with positive pairing. Consider the endomorphism
$d_{(F^\cdot),g}$ on $\CC[(K \oplus K^\dual)_0]\otimes \Lambda^*M_\CC$ 
defined by 
\begin{equation}\label{chiraldiff}
\sum_{i=0}^4 \sum_{m\in \Delta}F^i_m [m]\otimes (\wedge m_i)
+\sum_{n\in\Delta^\dual} g_n [n]\otimes ({\rm contr.} n).
\end{equation}
It is a differential by a calculation similar to Proposition \ref{diff}.

\begin{theorem}
For generic $F^\cdot$ and $g$ the eigenvalues of $H_A$  and $H_B$
on $V_{(F^\cdot),g}$ are nonnegative integers.
the $H_A=0$ part is given as the cohomology of the 
corresponding eigenspace of ${\rm Fock}_{K\oplus K^\dual -\deg^\dual}$.
As a vector space, this is isomorphic to the cohomology
of $\CC[(K\oplus K^\dual)_0]\otimes \Lambda^*M_\CC$ by 
$d_{(F^\cdot),g}$ from \eqref{chiraldiff}.
The $H_B=0$ part comes from the corresponding eigenspace of 
${\rm Fock}_{K-\deg \oplus K^\dual}$.
As a vector space it is isomorphic to the cohomology of 
$\CC[(K\oplus K^\dual)_0]\otimes \Lambda^*N_\CC$ by an operator
similar to \eqref{chiraldiff} where one replaces all wedge products 
by contractions and vice versa.
\end{theorem}

\begin{proof}
One follows the argument of \cite{chiralrings}.
\end{proof}

\begin{remark}
It would be interesting to compare this description of chiral rings to
other known statements about the (0,2) theories, see for example \cite{Guffin}.
It also appears that the work of \cite{PlesserMelnikov} is closely related to 
this paper.
\end{remark}

\begin{remark}
It is possible, in the quintic case, to completely calculate the products
in the chiral rings. However, this will be in the set of coordinates that
is somewhat different from the usual K\"ahler parameters.
We plan to return to this topic in future research. 
\end{remark}

\section{Concluding comments}\label{last}
The main philosophical outcome of this paper is a simple observation that
(0,2) string theory in toric setting (at the level of half-twisted theory) is
quite amenable to explicit calculations. The quintic case is however somewhat
special, because one is dealing with a smooth ambient variety.

\smallskip
The most general possible toric framework to which one can hope to extend 
this setup should also combine the almost dual Gorenstein cones explored in
\cite{BB-BH}. From this perspective the most generic ansatz that we wish
to make is the following.

\smallskip
Consider dual lattices $M$ and $N$ with elements $\deg\in M$ and $\deg^\dual\in N$.
Consider \emph{subsets} $\Delta$ and $\Delta^\dual$ in $M$ and $N$ respectively
with the properties 
$$
\Delta\cdot\deg^\dual = \deg \cdot \Delta^\dual=1, ~~\Delta\cdot\Delta^\dual\geq 0.$$
In addition, the cones generated by $\Delta$ and $\Delta^\dual$ should be almost
dual to each other, in some sense. 
It is possible that the technical definition of \cite{BB-BH} would 
still be appropriate, but since it might not be, we feel that it may not be wise to
present it here. 

\smallskip
Consider the lattice vertex algebra ${\rm Fock}_{M\oplus N}$.
Pick a basis $m_i$ of $M$ and $n_i$ of $N$.
Then one needs to consider collections of complex numbers
$F^i_m$ and $G^i_n$ for all $i$, $m\in\Delta$, $n\in \Delta^\dual$ such that 
the operator
$$
D_{(F^\cdot), (G^\cdot)}=
{\rm Res}_{z=0}\Big(\sum_{i} \sum_{m\in \Delta}F^i_m m_i^{ferm}(z)\ee^{\int m^{bos}(z)}
$$
$$+\sum_{i}\sum_{n\in\Delta^\dual} G^i_n n_i^{ferm}(z)\ee^{\int n^{bos}(z)}\Big)
$$
is a differential on ${\rm Fock}_{M\oplus N}$ (and in fact we want the OPE
of the above field with itself to be nonsingular). 

\smallskip
Then we would like to consider the cohomology of ${\rm Fock}_{M\oplus N}$
by the above differential. The hope is that under some almost duality condition 
the Key Lemma of \cite{Borvert} still works and we can then show that this
cohomology satisfies $H_A,H_B\geq 0$. 

\smallskip
It is not clear what, if any, geometric meaning one would be able to ascribe to
a generic family of algebras obtained in this fashion, but they appear to be very natural constructs to
study. In this context the (0,2) mirror symmetry would simply correspond to a 
switch between $M$ and $N$.


\begin{thebibliography}{999999}

\bibitem[B1]{Borvert}
L. Borisov, \emph{Vertex algebras and mirror symmetry.} 
Comm. Math. Phys. 215 (2001), no. 3, 517--557.

\bibitem[B2]{chiralrings}
L. Borisov, \emph{ Chiral rings of vertex algebras of mirror symmetry.}
Math. Z. 248 (2004), no. 3, 567--591.

\bibitem[B3]{BB-BH} 
L. Borisov, \emph{Berglund-H\"ubsch mirror symmetry via vertex algebras.}\\
preprint arXiv:1007.2633.

\bibitem[EHKZ]{Heluani}
J. Ekstrand, R. Heluani, J. K\"all\'en, M. Zabzine, 
\emph{ Non-linear sigma models via the chiral de Rham complex.}
Adv. Theor. Math. Phys. 13 (2009), no. 4, 1221â-€“1254. 

\bibitem[GM]{Malikov-Gorbounov}
V. Gorbounov, F. Malikov,
\emph{Vertex algebras and the Landau-Ginzburg/ Calabi-Yau correspondence.}
Mosc. Math. J. 4 (2004), no. 3, 729Ğ779, 784.

\bibitem[GMS1]{gerbes1} 
V. Gorbounov, F. Malikov,  V. Schechtman,
\emph{Gerbes of chiral differential \\ operators.} 
Math. Res. Lett. 7 (2000), no. 1, 55Ğ66. 

\bibitem[GMS2]{gerbes2}
V. Gorbounov, F. Malikov,  V. Schechtman,
\emph{Gerbes of chiral differential operators. II. Vertex algebroids.} 
Invent. Math. 155 (2004), no. 3, 605Ğ680.

\bibitem[GMS3]{gerbes3}
V. Gorbounov, F. Malikov,  V. Schechtman,
\emph{Gerbes of chiral differential operators. III.} The orbit method in geometry and physics (Marseille, 2000), 73Ğ100, Progr. Math., 213, Birkh\"auser Boston, Boston, MA, 2003.

\bibitem[Gu]{Guffin}
J. Guffin, \emph{Quantum Sheaf Cohomology, a pr\'ecis.} preprint arXiv:1101.1305.

\bibitem[KW]{WK}
A. Kapustin, E. Witten,
\emph{Electric-magnetic duality and the geometric Langlands program.}
Commun. Number Theory Phys. 1 (2007), no. 1, 1Ğ236

\bibitem[MeP]{PlesserMelnikov} 
I.V. Melnikov, M.R. Plesser
\emph{A (0,2) Mirror Map.} preprint \\
arXiv:1003.1303.

\bibitem[MSV]{MSV} 
F. Malikov,  V. Schechtman, A. Vaintrob, 
\emph{Chiral de Rham complex.} Comm. Math. Phys. 204 (1999), no. 2, 439Ğ473.

\bibitem[T]{Tan} M.-C. Tan,
\emph{Two-dimensional twisted sigma models, the mirror chiral de Rham complex, and twisted generalized mirror symmetry.}
J. High Energy Phys. 2007, no. 7, 013, 80 pp.

\bibitem[W]{Witten} E. Witten, \emph{Phases of N=2 theories in two dimensions.}
Nuclear Phys. B 403 (1993), no. 1-2, 159Ğ222.
 
\end{thebibliography}
\end{document}